\documentclass{amsproc}
\usepackage{bm, amssymb}
\usepackage{mathrsfs}
\usepackage{graphicx,lastpage}
\usepackage{upgreek}
\usepackage{censor}
\usepackage{amsmath, amsfonts, amssymb, mdwlist}
\usepackage{enumitem}
\usepackage{mathrsfs}
\usepackage{mathtools}
\usepackage{amsmath}
\usepackage{amssymb}
\usepackage{amsthm}
\usepackage{graphics}
\usepackage{eucal}
\usepackage{mathrsfs}
\usepackage{textcomp}
\usepackage{color}
\newtheorem{theorem}{Theorem}[section]
\newtheorem{proposition}{Proposition}[section]
\newtheorem{lemma}[theorem]{Lemma}
\newcommand\degree{\operatorname{deg}}
\DeclareMathOperator{\GL}{GL}

\theoremstyle{definition}
\newtheorem{definition}[theorem]{Definition}

\theoremstyle{remark}

\newcommand{\frk}[2][n]{\mathfrak{#2}}

\newcommand{\mcl}[2][n]{\mathcal{#2}}

\newcommand{\defeq}{\stackrel{\text{def}}{=}}

\newcommand\rank{\operatorname{rank}}

\newcommand{\diag}{\operatorname{diag}}

\numberwithin{equation}{section}

\begin{document}

\title[Diophantine approximation in positive characteristic]{Khintchine's theorem for affine hyperplanes in positive charateristic}

\author{Arijit Ganguly}
\address{Department of Mathematics and Statistics, IME Building
Indian Institute of Technology Kanpur
Kanpur, P.O.- IIT Kanpur, P.S.- Kalyanpur,
District - Kanpur Nagar,January 1, 1994 and, in revised form, June 22, 1994.
Pin - 208016,
Uttar Pradesh, India.}
\email{arijit.ganguly1@gmail.com}
\thanks{The author is grateful to Anish Ghosh for drawing his attention to this problem and all helpful discussions}

%    General info
\date{\today}


\begin{abstract} In this paper we establish the convergence case of Khintchine's theorem for affine hyperplanes in function field of positive characteristic. Along with that, we also prove a quantitative version of the same. The main technique used in the proof is a dynamical result called `Quantitative nondivergence' due to D. Y. Kleinbock and G. A. Margulis \cite{KM}. 

\end{abstract}
\maketitle
\section{Introduction} Diophantine approximation in the setting of local field with positive characteristic deals with the approximation of Laurent series in $T^{-1}$ by rational function in $T$, where $T$ is an indeterminate, and its higher dimensional analogues. This subject has drawn a considerable amount attention of mathematicians in last few decades and hence has been studied extensively. Indeed the geometry of numbers has been developed in this setting by K. Mahler in \cite{Mahler} in 1940, and consequently the Dirichlet type theorems can be established in this context (for an elementary proof of the most general and multiplicative form of Dirichlet's theorem in function fields, the reader is suggested to look at \cite{GG}). V. Sprind\v{z}uk proved the  analogue of Mahler's conjectures and some transference principles for function fields (see \cite{Spr1}). The reader is being referred to \cite{deM, Las1} for general surveys and also to \cite{AGP, Kris, KN, GR, GG2} (for a necessarily incomplete set of references) to see the state of the art of some of the recent developments.  \\

 Before we turn to the metrical aspect of the theory over positive characteristic, we would like to provide the reader with a very brief overview of the metric Diophantine approximation on manifolds. It is a branch of Number theory where one studies the extent the properties of a generic point in $\mathbb{R}^n$ with respect to Lebesgue measure or some nice measures are inherited by embedded submanifolds. Its genesis lies in the conjecture by K. Mahler in 40's that says, almost every points on the following curve, known as the \emph{Veronese curve}, \[\mathcal{V}_n\defeq \{(x,x^2, \dots, x^n):x\in \mathbb{R}\}\] are not \emph{very well approximable}, where the measure under consideration is the obvious push-forward of the Lebesgue measure on $\mathbb{R}$.  This conjecture has been settled by V. Sprind\v{z}uk later on, who in turn conjectured two more generalizations of that of Mahler. Those conjectures have remained unsolved for several decades, although some partial results have been found, until the attack by D. Y. Kleinbock and G. A. Margulis \cite{KM}. In this landmark work \cite{KM}, the authors have shown that if a point in a submanifold of $\mathbb{R}^n$ is very well approximable then the trajectory of some point, which is naturally associated to the given point in the manifold,  in the space of unimodular lattices will have a particular nature of visiting the cusp under the diagonal flow, then they use a sharp quantitative estimate on nondivergence of trajectories to show such orbits are really rare. The above mentioned `Quantitative nondivergence' estimate did undergo some further extensions and generalizations etc. in the subsequent works \cite{Kleinbock-extension}, \cite{BKM} and \cite{KT}. In fact, in \cite{BKM} Bernik, Kleinbock and Margulis used their generalized version of `Quantitative nondivergence' to prove the convergence case Khintchine's theorem in nondegenerate submanifolds of euclidean spaces. For affine hyperplanes and more generally for affince subspaces and their nondegenrate submanifolds, the same have been established by Anish Ghosh in \cite{G-hyper} and \cite{G-monat} respectively. Effective version of Khintchine theorem for nondegenerate euclidean submanifolds are provided by Adiceam et.al. \cite{Ad}, and for affinbe subspaces by Arijit Ganguly and Anish Ghosh \cite{GG-quant}. For a survey of some of the recent developments in Diophantine approximation in subspaces of $\mathbb{R}^n$ and its connection with homogeneous dynamics see \cite{G-survey}. \\
 
 Coming back to the positive characteristic theory, although there are many interesting parallels with that of over $\mathbb{R}$, it often offers surprises too. Many results 
 hold in both settings, the main results of the paper being of such kind, while there are noteworthy exceptions in place. To mention a few, the theory of badly approximable 
 numbers and vectors in positive characteristic is quite different, there is no analogue of Roth's theorem, provided that the base field is finite, and last but not the least, the Dirichlet improvable Laurent series are precisely all rational functions (see \cite[Theorem 2.4]{GG2}) unlike the scenario over real numbers where for an irrational number, Dirichlet improvability is equivalent to being badly approximable (see \cite{DS}). \\
 
The conjectures of  V. Sprind\v{z}uk in this seeting have been proved by Anish Ghosh in \cite{G-pos}. But the Khintchine theorem in this context have not been studied much to date to the best of author's knowledge. The well approximability of matrices over field of Laurent series have been first studied by Simon Kristensen in \cite{Kris1}. Indeed, the author proved a Khintchine type theorem and examined the exceptional set in terms of Hausdorff dimension. Proceeding one step further, one can ask the the analogues question for \emph{nonplanar submanifolds} in this setting, i.e., images under \emph{nonplanar} maps. This problem seems to be still open. \\
 
 The aim of this paper is to prove the analogue of Khinchine theorem for affine hyperplances in positive characteristic and furthermore provide an effective version of that as well. We mainly proceed along the lines of \cite{G-hyper, GG-quant} and make necessary adaptations of those techniques in this context. Indeed, this paper can be regarded as the positive characteristic version of \cite{G-hyper}, and to the extent of affine hyperplanes, the same can be said for \cite{GG-quant}. It is quite plausible that all the results of this paper can be generalized to any affine subspace. Indeed, one can expect a suitable analogue of the hypothesis on the higher Diophantine exponents mentioned in \cite{Kleinbock-extension, G-monat, GG-quant} should work in this case generalizing \eqref{cond} of this paper. However, we refrain ourselves from that mainly because of the above mention Diophantine condition on higher exponents is not so well understood for higher codimensional subspaces except lines. To elaborate a little more, given an affine  subspace of $\mathbb{R}^n$ with dimension strictly lying between $1$ and $n-1$, neither it is clear how to interpret that nor one knows how to verify whether the condition holds for that subspace. But for codimension $1$, it happens to be a Diophantine property of the vector  defining  the affine hyperplane. 
\subsection{The set up}
 Let $p$ be a prime and $q:= p^a$, where $a\in \mathbb{N}$, and consider the function 
field $\mathbb{F}_{q}(T)$. We define a function $|\cdot|: \mathbb{F}_{q}(T) \longrightarrow \mathbb{R}_{\geq 0}$ as follows. 
\[ |0|\defeq 0\,\,  \text{ and} \,\, \left|\frac{P}{Q}\right|\defeq q^{\displaystyle \degree P- \degree Q}
\text{ \,\,\,for all nonzero } P, Q\in \mathbb{F}_{q}[T]\,.\] 
Clearly $|\cdot|$ is a nontrivial,  non-archimedian and discrete absolute value  
in $\mathbb{F}_{q}(T)$. This absolute value gives rise to a metric on $\mathbb{F}_{q}(T)$. \\

The completion field of $\mathbb{F}_{q}(T)$ is $\mathbb{F}_{q}((T^{-1}))$, i.e. the field of Laurent series 
over $\mathbb{F}_{q}$. The absolute value of $\mathbb{F}_{q}((T^{-1}))$, which we again denote by $|\cdot |$, is given as follows. 
Let $a \in \mathbb{F}_{q}((T^{-1}))$. For $a=0$, 
define $|a|=0$. If $a \neq 0$, then we can write 
$$a=\displaystyle \sum_{k\leq k_{0}} a_k T^{k}\,\,\mbox{where}\,\,\,\,k_0 \in \mathbb{Z},\,a_k\in \mathbb{F}_{q}\,\,\mbox{and}\,\, a_{k_0}
\neq 0\,. $$
\noindent We define $k_0$ as the \textit{degree} of $a$, which will be denoted by $\degree a$,  and     
$|a|:= q^{\degree a}$. This clearly extends the absolute 
value $|\cdot|$ of $\mathbb{F}_{q}(T)$ to $\mathbb{F}_{q}((T^{-1}))$ and moreover, 
the extension remains non-archimedian and discrete. Let $\Lambda$ and $F$ 
denote $\mathbb{F}_{q}[T]$ and $\mathbb{F}_{q}((T^{-1}))$ respectively from now on. It is obvious that 
$\Lambda$ is discrete in  $F$. For any $d\in \mathbb{N}$, $F^d$ is throughout assumed to be equipped 
with the supremum norm which is defined as follows
\[||\mathbf{x}||\defeq \displaystyle \max_{1\leq i\leq n} |x_i|,\text{ for all } \mathbf{x}=(x_1,x_2,...,x_d)\in F^{d}\,,\]

\noindent and with the topology induced by this norm. Clearly $\Lambda^n$ is discrete in $F^n$. Since the topology on $F^n$ considered here 
is the usual product topology on $F^n$, it follows that  $F^n$ is locally compact as $F$ is locally compact. Let $\lambda$ be the Haar measure on $F^n$ which takes the value 1 on the closed unit ball $||\mathbf{x}||=1$. Note that, by a routine combinatorial argument it is easy to see that for any $t\in  \mathbb{N}$, we have $$\#\{\mathbf{q}\in \Lambda^n: ||\mathbf{q}||=q^t\}=q^{(t+1)n}-q^{tn}= q^{nt}(q^n-1).$$
\subsection{Main results}
Let $\mathscr{H}$ stand for the following hyperplane in what follows: $$\{(\mathbf{x}, \tilde{\mathbf{x}}\cdot\mathbf{a}): \mathbf{x}\in F^{n-1}\},$$ where $\mathbf{a}:=(\alpha_0, \alpha_1, \dots, \alpha_{n-1})\in F^n$, and $\tilde{\mathbf{x}}:= (1, \mathbf{x})$. Assume that $$\psi:\{q^r: r\in \mathbb{Z}_{\geq 0}\}\longrightarrow \{q^r: r\in \mathbb{Z}\}$$ is a nonincreasing approximation function. We say $\mathbf{y}\in \mathscr{H}$ is  \emph{$\psi$-approximable} if, for infinitely many $\mathbf{q} \in \Lambda^n$, one has the following: $$|\mathbf{y}\cdot \mathbf{q}+p|< \psi(||\mathbf{q}||), \text{ for some }p\in \Lambda.$$ The set of all $\psi$-approximable points in $\mathscr{H}$ will be denoted by $\mathcal{W}({\mathscr{H}};\psi)$. We prove the version of classical Khintchine's theorem in the context of affine hyperplanes in positive characteristic under reasonably relaxed hypothesis. Our fist main theorem would be:
\begin{theorem}\label{main}
	Let the set up be as above. Assume that, there exists $\delta >0$ such that, for all but finitely many $q'\in \Lambda$ the following holds:
	\begin{equation}\label{cond}
	\max_{0\leq i\leq n-1}|p_i+\alpha_iq'|>\frac{1}{|q'|^{n-\delta}},
	\end{equation} for all $\mathbf{p}\in \Lambda^n$. Then one has $\lambda(\mathcal{W}({\mathscr{H}};\psi))=0$ whenever $$\displaystyle\sum_{t=0}^{\infty} \psi(q^t)q^{nt}\asymp \sum_{\mathbf{q}\in \Lambda^n}\psi(||\mathbf{q}||)<\infty.$$
\end{theorem}
Consider an open ball $U\subseteq F^{n-1}$. It follows from Theorem \ref{main} that, for almost all $\textbf{x}\in F^{n-1}\in U$, one can find $\kappa>0$ such that 
\begin{equation}\label{kappa}
|(\mathbf{x}, \tilde{\textbf{x}}\cdot\textbf{a})\cdot \mathbf{q}+p|\geq \kappa \psi(||\mathbf{q}||), \text{ for all }\textbf{q}\in \Lambda^n\setminus \{\textbf{0}\} \text{ and }p\in \Lambda,
\end{equation} provided $\sum_{\psi}\defeq \sum_{\mathbf{q}\in \Lambda^n}\psi(||\mathbf{q}||)<\infty$. We are interested to obtain this $\kappa$ independent of $\textbf{x}$ to the extent possible in the sense of measure. In order to emphasize this, let us consider the following subset 
\begin{equation}\label{set}
\mathcal{B}(U,\psi,\kappa)\defeq \{\textbf{x}\in U: \eqref{kappa}\text{ holds}\}.
\end{equation}We aim to see dependence between $\kappa$ and the size of the set \eqref{set}. Our second main theorem quantifies that as follows: 
\begin{theorem}\label{mainquant}
	Let the hypothesis be as in Theorem \ref{main}. Assume further that $\psi(q^r)\leq \frac{1}{q^{rn}}$, for all $r\geq 0$. Fix an open ball $U\subseteq F^{n-1}$. Then there exists two explicitly computable constants $K_0$ and $K_1$, depending upon $n, U$ and $\mathbf{a}$ only, with the following property: for any $\xi\in (0,1)$, the following holds 
	\begin{equation}
	\lambda(\mathcal{B}(U,\psi,\kappa))\geq (1-\xi)\lambda(U)
	\end{equation} with any  $$\kappa<\min \left\{1, \frac{\xi}{2q\sum_{\psi}}, \left(\frac{\xi}{2K_0K_1}\right)^{n^2-1}\right\}.$$ 
\end{theorem}
In order to prove the above theorems, we adopt the general strategy of \cite{G-hyper, GG-quant}, which is originally  invented in \cite{BKM}, namely dividing into two cases depending upon big and small gradients. 
	\section{The gradient division}
	For any $\mathbf{q}\in \Lambda^n\setminus\{\mathbf{0}\}$ and $\kappa>0$, we define the following sets:\[\mathcal{L}(\mathbf{q})\defeq \{\mathbf{x}\in U: |(\mathbf{x}, \tilde{\mathbf{x}}\cdot\mathbf{a})\cdot\mathbf{q}+p|<\psi(||\mathbf{q}||) \text{ for some }p\in \Lambda \}\] and \[\mathcal{L}(\mathbf{q}, \kappa)\defeq \{\mathbf{x}\in U: |(\mathbf{x}, \tilde{\mathbf{x}}\cdot\mathbf{a})\cdot\mathbf{q}+p|<\kappa \psi(||\mathbf{q}||) \text{ for some }p\in \Lambda \}.\]
	
It is clear that $$\displaystyle \limsup_{\mathbf{q}\in \Lambda^n\setminus \{\mathbf{0}\}}\mathcal{L}(\mathbf{q})=\mathcal{W}({\mathscr{H}};\psi)\cap U \text{ and }\displaystyle \bigcup_{\mathbf{q}\in \Lambda^n\setminus \{\mathbf{0}\}}\mathcal{L}(\mathbf{q}, \kappa)=U\setminus\mathcal{B}(U,\psi,\kappa).$$ For any $t\in \mathbb{N}\cup \{0\}$, let $\mathcal{L}^{<}_t$ and $\mathcal{L}^{\geq}_t$ stand for the following sets
\[\displaystyle \left \{\mathbf{x}\in U: \exists \mathbf{q}=(q_1,\dots,q_n)\in \Lambda^n\setminus \{\mathbf{0}\} \text{ and }p\in \Lambda \text{ s.t. }\left | \begin{array}{rcl} |(\mathbf{x}, \tilde{\mathbf{x}}\cdot\mathbf{a})\cdot\mathbf{q}+p|< \psi(q^t)  \\\forall 1\leq i\leq n-1,\,|q_i+\alpha_iq_n|<1 \\ ||\mathbf{q}||=q^t
\end{array}\right. \right\}\] and 
\[\displaystyle \left \{\mathbf{x}\in U: \exists \mathbf{q}=(q_1,\dots,q_n)\in \Lambda^n\setminus \{\mathbf{0}\} \text{ and }p\in \Lambda \text{ s.t. }\left | \begin{array}{rcl} |(\mathbf{x}, \tilde{\mathbf{x}}\cdot\mathbf{a})\cdot\mathbf{q}+p|< \psi(q^t)  \\ \exists i\in \{1,\cdots,n-1\} \text{ s.t. }|q_i+\alpha_iq_n|\geq 1 \\ ||\mathbf{q}||=q^t
\end{array}\right. \right\}\] respectively. Now observe that, in order to show $\lambda(\mathcal{W}({\mathscr{H}};\psi)\cap U)=0$, it will be enough to show the convergence of $\displaystyle \sum_{t=0}^{\infty}\lambda(\mathcal{L}^{<}_t)$ and $\displaystyle \sum_{t=0}^{\infty}\lambda(\mathcal{L}^{\geq}_t)$.\\

To deal with the quantitative part, define similarly, 
\[\mathcal{L}_t^{<}(\kappa)\defeq \displaystyle \left \{\mathbf{x}\in U: \exists \mathbf{q}\in \Lambda^n\setminus \{\mathbf{0}\} \text{ and }p\in \Lambda \text{ s.t. }\left | \begin{array}{rcl} |(\mathbf{x}, \tilde{\mathbf{x}}\cdot\mathbf{a})\cdot\mathbf{q}+p|< \kappa \psi(q^t)  \\\forall 1\leq i\leq n-1,\,|q_i+\alpha_iq_n|<1 \\ ||\mathbf{q}||=q^t
\end{array}\right. \right\}\] and 
\[\mathcal{L}_t^{\geq}(\kappa)\defeq \displaystyle \left \{\mathbf{x}\in U: \exists \mathbf{q}\in \Lambda^n\setminus \{\mathbf{0}\} \text{ and }p\in \Lambda \text{ s.t. }\left | \begin{array}{rcl} |(\mathbf{x}, \tilde{\mathbf{x}}\cdot\mathbf{a})\cdot\mathbf{q}+p|< \kappa \psi(q^t)  \\ \exists i\in \{1,\cdots,n-1\} \text{ s.t. }|q_i+\alpha_iq_n|\geq 1 \\ ||\mathbf{q}||=q^t
\end{array}\right. \right\}.\] for all $t\in \mathbb{N}\cup \{0\}$. We prove that, for suitably chosen $\kappa$, one has the following 
\begin{equation}\label{quant small}\displaystyle \sum_{t=0}^{\infty}\lambda(\mathcal{L}^{<}_t(\kappa))<\frac{\xi}{2}\lambda(U),
\end{equation} and 
\begin{equation}\label{quant large}\displaystyle \sum_{t=0}^{\infty}\lambda(\mathcal{L}^{\geq}_t(\kappa))<\frac{\xi}{2}\lambda(U)
\end{equation} as clearly \eqref{quant small} and \eqref{quant large} together implies $\lambda\left(\displaystyle \bigcup_{\mathbf{q}\in \Lambda^n\setminus \{\mathbf{0}\}}\mathcal{L}(\mathbf{q}, \kappa)\right)< \xi \lambda(U)$.
\section{Estimating  $\lambda(\mathcal{L}^{\geq}_t)$ and $\lambda(\mathcal{L}^{\geq}_t(\kappa))$}\noindent We need the following general Proposition: 
\begin{proposition} \label{big_prop}Fix $\mathbf{q}=(q_1, \dots, q_n)\in \Lambda^n$ be such that 
	\begin{equation}\label{large} |q_i+\alpha_iq_n|\geq 1, \text{ for some }1\leq i\leq n-1
	\end{equation} For any $m\in \mathbb{N}$, consider the set
\begin{equation}\label{Prop}
\left \{\mathbf{x}\in U: |(\mathbf{x}, \tilde{\mathbf{x}}\cdot\mathbf{a})\cdot\mathbf{q}+p|<\frac{1}{q^m}, \text{ for some }p\in \Lambda \right \}. 
\end{equation}Then the set defined above in \eqref{Prop} has measure $\displaystyle \leq \frac{q\lambda(U)}{q^m}$. 
		\end{proposition}
	\begin{proof}
		 Observe that, for any $\mathbf{x}=(x_1,\dots, x_{n-1}) \in F^{n-1}$, 
	\begin{eqnarray*}
		(\mathbf{x}, \tilde{\mathbf{x}}\cdot\mathbf{a})\cdot\mathbf{q}+p&=(q_1+\alpha_1q_n)x_1+\cdots+(q_{n-1}+\alpha_{n-1}q_n)x_{n-1}+\alpha_0q_n+p\\ & = \beta_1x_1+\cdots+\beta_{n-1}x_{n-1}+y, 
		\end{eqnarray*}
	where $\beta_i\defeq q_i+\alpha_iq_n$, for $i=1,\dots,n-1$ and $y\defeq \alpha_0q_n+p$. Let $\beta\defeq (\beta_1, \dots, \beta_{n-1})$ and $\|\beta\|=\beta_i$. Clearly, from \eqref{large}, one has $|\beta_i|\geq 1$. Also $U$ is of the form $U_1\times \cdots \times U_{n-1}$, where $U_1,\dots, U_{n-1}$ are one dimensional balls having the same radius with that of $U$.  Now the set defined in \eqref{Prop} is the union of some strips of the following type:
	\[\mathscr{S}_p\defeq\left \{(x_1,\dots,x_{n-1})\in U_1\times \cdots\times U_{n-1}: |\beta_1x_1+\cdots+\beta_{n-1}x_{n-1}+y|<\frac{1}{q^m}\right \}.\] We  apply Fubini's theorem to estimate $\lambda(\mathscr{S}_p)$ as follows: 
	 \begin{eqnarray*}\displaystyle \int_{F^{n-1}}\mathbf{1}_{\mathscr{S}_p}(x_1,\dots, x_{n-1})\,dx_1\,\dots\,dx_{n-1}\\ = \displaystyle \int_{U_1\times \cdots \times U_{i-1}\times U_{i+1}\times \cdots \times U_n}\lambda \left(\left\{x_i\in B_i: \left|\beta_ix_i+ \left(\displaystyle \sum_{j\neq i} \beta_jx_j+ y\right)\right|<\frac{1}{q^m}\right\}\right)\,\prod_{j\neq i} dx_j \\ \displaystyle \leq \frac{r(U)^{n-2}}{|\beta_i|q^m} \displaystyle \leq \frac{r(U)^{n-2}}{\|\beta\|q^m},
	 		\end{eqnarray*} where $r(U)$ stands for the radius of the ball $U$. Observe that, if $p_1\neq p_2\in \Lambda$ are such that $\mathscr{S}_{p_1}, \mathscr{S}_{p_2}\neq \emptyset$ then \[1\leq |p_1-p_2|\leq \|\beta\|r(U). \] From this, it follows that number of such nonempty strips is at most $q\|\beta\|r(U)$. Hence, the measure of the set defined in \eqref{Prop} is at most $$q\|\beta\|r(U)\times \displaystyle  \frac{r(U)^{n-2}}{\|\beta\||q^m}=\frac{q\lambda(U)}{q^m}.$$
 		\end{proof}
 	In view of this, one obtains from Proposition \ref{big_prop} that, 
 	\[\lambda(\mathcal{L}^{\geq}_t)\leq q\lambda(U) (q^n-1) \psi(q^t)q^{nt}, \,\forall t\in \mathbb{N}\]This leads us to the following desired convergence:
 	$$\displaystyle \sum_{t=0}^{\infty}\lambda(\mathcal{L}^{\geq}_t)\leq \displaystyle q\lambda(U)\sum_{t=0}^{\infty} \psi(q^t)q^{nt}(q^n-1)=q\lambda(U)\sum_{\mathbf{q}\in \Lambda^n}\psi(||\mathbf{q}||)<\infty.$$
 	In the similar manner, one also obtains from Proposition \ref{big_prop} that, for any $t
 \in \mathbb{N}$, \[\lambda(\mathcal{L}^{\geq}_t(\kappa))\leq q\lambda(U)(q^n-1) \kappa \psi(q^t)q^{nt},\]  and thus, it is now immediate that, 
 	$$ \sum_{t=0}^{\infty}\lambda(\mathcal{L}^{\geq}_t(\kappa))\leq  q\lambda(U)\kappa \sum_{t=0}^{\infty} \psi(q^t)q^{nt}(q^n-1)=q\lambda(U)\kappa \textstyle \sum_{\psi}=\kappa q \textstyle \sum_{\psi}\lambda(U)<\frac{\xi}{2}\lambda(U).$$
 	\indent We now turn  to the estimates of $\lambda(\mathcal{L}_t^{<}))$ and $\lambda(\mathcal{L}_t^{<}(\kappa))$. For this purpose, we use the `Quantitative nondivergence' estimate, which has originally been invented by  D. Y. Kleinbock and G. A. Margulis in \cite{KM} in the space of all unimoduloar lattices in euclidean spaces, and subsequently, extended, generalised by several authors.  Specifically to this context, we need a slight generalisation of \cite[Theorem 7.3]{KT}, which can be regarded as the function field version of \cite[Theorem 6.2]{BKM} and yields us $ \sum_{t=1}^{\infty}\lambda(\mathcal{L}^{<}_t)<\infty$ and \eqref{quant small}. Actually,  behind the proof of any of the quantivative nondivergence results mentioned above, there lies a subtle analysis of maps of posets into the spaces of `good functions'. Hence, to begin with, it is necessary to recall the basic properties of `good functions'. This is the objective of following short section consisting of the statements of the some results without proof. 
 	\section{Good functions and their basic properties}
 	D. Y. Kleinbock and G. A. Margulis introduced the notion of a `good function' in \cite{KM}  and later on, it has been generalized to any Besicovicth space by Kleinbock and Tomanov in \cite{KT}. Here we follow  \S's 1 and 2 of \cite{KT} closely. For the sake of generality, we assume $X$ is a Besicovitch metric space, $U\subseteq X$ is  open, $\nu$ is a radon measure
 	on $X$, $(\mcl{F},|\cdot|)$ is a valued 
 	field and $f: X\longrightarrow\mcl{F}$ is a given function such that $|f|$ is measurable. For any $B\subseteq X$, we set
 	$$ ||f||_{\nu,B} := \displaystyle \sup_{x\in B\cap \text{ supp }(\nu)} |f(x)|.$$
 	
 	\begin{definition}\label{defn:C,alpha}
 		For $C,\alpha \textgreater\,0$, $f$ is said to be $(C,\alpha)-good$ on $U$ with respect to $\nu$ if for every ball $B\subseteq U$
 		with center in $\text{supp }(\nu)$, one has
 		\[\nu(\{x\in B: |f(x)|\,\textless \varepsilon\})\leq C\left(\frac{\varepsilon}{||f||_{\nu,B}}\right)^{\alpha} \nu(B)\,.\]
 	\end{definition}
 	The following properties are immediate from Definition \ref{defn:C,alpha}. 
 	\begin{lemma} \label{lem:C,alpha} Let $X,U,\nu, \mcl{F}, f, C,\alpha,$ be as given above. Then one has 
 		\begin{enumerate}
 			\item $f$ \text{ is } $(C,\alpha)-good$ \text{ on  }$U$ with respect to $\nu \Longleftrightarrow \text{ so is } |f|$.
 			\item $f$ is $(C,\alpha)-good$ on $U$  with respect to $\nu$ $\Longrightarrow$ so is $c f$ for all $c \in \mcl{F}$.
 			\item \label{item:sup} $\forall i\in I, f_i$ are $(C,\alpha)-good$ on $U$ with respect to $\nu$ and $\sup_{i\in I} |f_i|$ is measurable $\Longrightarrow$ so is $\sup_{i\in I} |f_i|$.
 			\item $f$ is $(C,\alpha)-good$ on $U$ with respect to $\nu$ and $g :V \longrightarrow \mathbb{R}$ is a continuous function such 
 			that $c_1\leq |\frac{f}{g}|\leq c_2$ 
 			for some $c_1,c_2 \,\textgreater \,0\Longrightarrow g$ is $(C(\frac{c_2}{c_1})^{\alpha},\alpha)$ good on $U$ with respect to $\nu$.
 			\item Let $C_2 \,\textgreater \,1$ and 
 			$\alpha _2\,\textgreater\,0$. $f$ is $(C_1,\alpha_1)-good$ on $U$ with respect to $\nu$ and $C_1 \leq C_2, \alpha_2 \leq \alpha_1 \Longrightarrow f$ is
 			$(C_2,\alpha_2)-good$ on $V$ with respect to $\nu$.
 		\end{enumerate}   \end{lemma}
 	A prototype of good functions is polynomial functions over any local field. Indeed, the following has been proved in \cite{KT}:
 	\begin{proposition}[Lemma 3.4, \cite{KT}]\label{good prop}
 		Let $\mathcal{F}$ be $\mathbb{R}$ or a local field with an ultametric valuation. Then for any $d, k\in \mathbb{N}$ and $f\in \mathcal{F}[x_1,\dots,x_d]$ with $\deg f\leq k$, $f$ is $(C,1/dk)$-good with respect to $\lambda$, where $\lambda=$ Lebesgue measure if $\mathcal{F}=\mathbb{R}$ and $=$ the normalized Haar measure such that the volume of the unit ball is $1$ if $\mathcal{F}$ is ultrametric, and $C$ is a constant that depends only on $d$ and $k$. 
 	\end{proposition}
 	Our next section deals with the function field analogue of \cite[Theorem 6.2]{BKM} obtained by generalizing \cite[Theorem 7.2]{KT}. 
 	 	 \section{A generalization of `Quantitative nondivergence' estimate}
 	We adopt the setting of \cite{KT} in this section. To begin with, fix a metric space $X$, $B\subseteq X$, a poset $\mathfrak{P}$ and a map $\varphi$ from  $\mathfrak{P}$ to $C(B)$, i.e., the space of all continuous real valued functions on $B$, which we denote by $s\mapsto \varphi_s, \,\forall s\in \mathfrak{P}$. Given two positive numbers $\varepsilon$ and $\rho$ with $\varepsilon\leq \rho$, we say that a point $z\in B$ is $(\varepsilon, \rho)$-\emph{protected relative to }$\mathfrak{P}$ if one can find a totally ordered subset $\mathfrak{S}_z$ of $\mathfrak{P}$ with the following two properties:
 	\begin{enumerate} [label=(\alph*)]
 		\item \label{protected 1} $\forall s\in \mathfrak{S}_z, \varepsilon\leq |\varphi_s(z)|\leq \rho$; and 
 		\item \label{protected 2} $|\varphi_s(z)|\geq \rho$, for every $s\in \mathfrak{P}\setminus \mathfrak{S}_z$ comparable to every element of $\mathfrak{S}_z$.  
 	\end{enumerate}
  The set of all such points is denoted by $\Phi(\varepsilon, \rho, \frk{P})$ henceforth. \\
  
  For $m\in \mathbb{N}$ and a ball $B=B(x;r)\subseteq X$, where $x\in X$ and $r\,\textgreater\,0$, we shall use the notation
  $3^kB$, where $k\in \mathbb{N}$, to denote the ball $B(x;3^kr)$. A locally finite Borel measure $\nu$ on $X$ is said to be \emph{uniformly Federer}  if there exists $D>0$ such that, for any ball centrer in the support of $\mu$, one has 
   \[\frac{\nu(3B)}{\nu(B)}\leq D\,.\]
In that case, we set $$D_{\nu}\defeq \displaystyle \sup_{\text{ball }B\subseteq X \text{ centered in supp } \nu }\frac{\nu(3B)}{\nu(B)}.$$
	\begin{theorem}[Theorem 6.1, \cite{KT}]\label{theorem_poset}
   
  Assume that $X$ is a Besicovitch metric space, $\mu$ is a uniformly Federer measure on $X$, $k\in \mathbb{N}$, and $C, \alpha, \rho >0$. Let $\mathfrak{P}$ be some poset, $B\subseteq X$ be a ball and $\varphi: \mathfrak{P}\longrightarrow C(3^kB)$ such that we have the following:
  \begin{enumerate} [label=(\alph*)]
  	\item no linearly ordered subset of $\mathfrak{P}$ contains more than $k$ elements, 
  	\item $\forall s\in \mathfrak{P}$,   $\varphi_s$ is $(C,\alpha)$-good on $3^kB$ with respect to $\mu$, 
  	\item  $\forall s\in \mathfrak{P}$, $\|\varphi_s\|_{\mu, B}\geq \rho$, and 
  	\item $\forall y\in 3^k B\cap \text{supp }\mu$, $\#\{s\in \mathfrak{P}: |\varphi_s(y)|<\rho\}$ is finite. 
  \end{enumerate}
Then for every $\varepsilon>0$ with $\varepsilon\leq \rho$, one has 
\[\mu(B\setminus \Phi(\varepsilon, \rho, \mathfrak{P}) \leq kC(N_XD_{\mu}^2)^k \left(\frac{\varepsilon}{\rho}\right)^{\alpha}\mu(B), \] where $N_X$ is the Besicovitch constant for $X$. 
  \end{theorem}
We apply Theorem \ref{theorem_poset} to the poset of submodules of some fixed $\mathcal{D}$ submodule of $\mathcal{R}^m$, where  $\mathcal{D}$ is a PID, $\mathcal{K}$ is the quotient field of $\mathcal{D}$, $\mathcal{R}$ is an extension of the field $\mathcal{K}$ and $m\in \mathbb{N}$.  If $\Delta$ is a $\mathcal{D}$ submodule of $\mathcal{R}^m$, let us 
denote by $\mathcal{R}\Delta$ its $\mathcal{R}$-linear span
inside $\mathcal{R}^m$. If $\Theta$ is a  $\mathcal{D}$ submodule of $\mathcal{R}^m$ and $\Delta$ is a submodule of $\Theta$, we say that 
$\Delta$ \textit{is primitive in} $\Theta$ if $\Delta=\mathcal{R}\Delta \cap \Theta$, (equivalently, $\Delta=\mathcal{K}\Delta \cap \Theta$). We see that the 
set of all nonzero primitive submodules of a fixed $\mathcal{D}$ submodule $\Theta$ of $\mathcal{R}^m$, denoted by $\frk{P}(\Theta)$,
is a partially ordered set with respect to set inclusion and its length is equal to $\rank(\Theta)$. Note that for any submodule $\Delta'$ of $\Theta$, $\Delta \defeq \mathcal{R}\Delta'\cap \Theta$ is also a $\mathcal{B}$ submodule which is primitive in $\Theta$. Suppose that $\mathcal{R}$ has a valuation, denoted by $|\cdot|$, that makes $\mathcal{R}$ a local field. We consider the topological group
$\GL(m,\mathcal{R})$ of $m\times m$ invertible matrices with entires in $\mathcal{R}$. It is obvious that any $g\in \GL(m,\mathcal{R})$ maps 
$\mathcal{D}$ submodules of $\mathcal{R}^m$ to $\mathcal{D}$ submodules of $\mathcal{R}^m$ preserving their rank and
inclusion relation.\\

A function $\|\cdot \|: \bigwedge (\mathcal{R}^m)\longrightarrow \mathbb{R}_+$ is said to be \emph{submultiplicative} if it is continuous with respect to the natural topology of $\bigwedge (\mathcal{R}^m)$, homogeneous, i.e., $\|t\mathbf{w}\|=|t| \|\mathbf{w}\|, \,\forall \mathbf{w}\in \bigwedge (\mathcal{R}^m)$ and $t\in \mathcal{R}$, and for any $\mathbf{v}, \mathbf{w}\in \bigwedge (\mathcal{R}^m)$, one always has the inequality $\|\mathbf{v}\wedge \mathbf{w}\|\leq \|\mathbf{v}\|\cdot \|\mathbf{w}\|$. Now given a Besicovitch metric space $X$ equipped with a uniformly Federer measure $\mu$, a ball $B\subseteq X$, a  $\mathcal{D}$ submodule $\Theta$ of $\mathcal{R}^m$ having rank $k$, a continuous function $h: 3^kB \longrightarrow \GL_m(\mathcal{R})$ and a submultiplicative function $\|\cdot \|$, we consider  the function $\varphi: \frk{P}(\Theta)\longrightarrow C(3^k B)$ defined as follows: for $\Delta \in \frk{P}(\Theta)$, $\varphi_{\Delta}$ is the map 
$x\mapsto \|h(x)\Delta\|, \,\forall x\in 3^kB$. Clearly every such $\varphi_{\Delta}$ is a continuous function on $3^kB$. 
\begin{theorem}\label{qn_gen}
	Let $\mathcal{R}, k, m, \Theta, B, \|\cdot\|$ and $\varphi$ be as above. Assume that, $C, \alpha >0$ and $\rho\in (0,1]$ are constants such that one has the following:
	\begin{enumerate} [label=(\Alph*)]
		\item \label{good}$\forall \Delta \in \frk{P}(\Theta)$, $\varphi_{\Delta}$ is $(C, \alpha)$-good on $3^kB$, 
		\item \label{lower} $\forall \Delta \in \frk{P}(\Theta)$. $\|\varphi_{\Delta}\|_{\mu, B}\geq \rho$; and	\item \label{discrete} $\forall x\in 3^k B\cap \text{supp }\mu$, $\#\{\Delta\in \mathfrak{P}(\Theta):\|h(x)\Delta\|<\rho\}$ is finite. 
	\end{enumerate}Then for any positive $\varepsilon\leq \rho$, we have 
\[\mu(\{x\in B: \|h(x)\theta\|<\varepsilon, \text{ for some }\theta\in \Theta \setminus \{\mathbf{0}\}\})\leq kC(N_XD_{\mu}^2)^k \left(\frac{\varepsilon}{\rho}\right)^{\alpha}\mu(B).
\]
	\end{theorem}
\begin{proof}
	It is enough to show that, for every $0<\varepsilon\leq \rho$, and every $x\in B$ that is $(\varepsilon, \rho)$-protected relative to $\frk{P}(\Theta)$, one has the following:
	\[\|h(x)\theta\|\geq \varepsilon \text{ for all }\theta\in \Theta \setminus \{\mathbf{0}\}.\] Let $\{\mathbf{0}\}=\Delta_0\subsetneqq \Delta_1\subsetneqq \cdots \subsetneqq  \Delta_{\ell}=\Theta$ be all elements of $\frk{S}_x\cup \{\{\mathbf{0}\}, \Theta\}$. Pick any $\theta\in \Theta \setminus \{\mathbf{0}\}$. Then there exists a unique $i\in \{1,\dots,\ell\}$ such that $\theta\in \Delta_i \setminus \Delta_{i-1}$. Take $\Delta \defeq (\Delta_{i-1}+\mathcal{R}\theta)\cap \Theta$. It is clear from the construction that $\Delta$ is primitive in $\Theta$ and it is contained in $\Delta_i$. This makes $\Delta$ comparable with every element of $\frk{S}_x$. Now the submultiplicativity yields that $\|h(x)\Delta\|\leq \|h(x)\Delta_{i-1}\|\|h(x)\theta\|$. From \ref{protected 1} and \ref{protected 2}, it is obvious that $\|h(x)\Delta\|\geq \min (\varepsilon, \rho)=\varepsilon$. From this one concludes that 
	\[\|h(x)\theta\|\geq \frac{\|h(x)\Delta\|}{\|h(x)\Delta_{i-1}\|}\geq \frac{\varepsilon}{\rho}\geq \varepsilon.\] Now Theorem \ref{theorem_poset} applies and completes the proof of Theorem \ref{qn_gen}. 
\end{proof}
\section{Estimating  $\lambda(\mathcal{L}^{<}_t)$}\label{khinchine}We give an explicit estimate of $\lambda(\mathcal{L}^{<}_t)$, for sufficiently large $t$, in this section from which the convergence of $\displaystyle \sum_{t=0}^{\infty}\lambda(\mathcal{L}^{<}_t)$ will follow at once. The method adopted is similar to that of \cite{GG-quant, G-hyper}.\\

Let $\beta\in (0, 1/(n+1))\cap \mathbb{Q}$ whose choice will be specified later. By $\mathcal{R}$ we denote the finite extension of $F$ that contains $T^{\frac{1}{n+1}}$ and $T^{(\frac{1}{n+1}-\beta)}$. Such an extension always exists and indeed the normalized valuation of $F$ can be uniquely extended to a normalized valuation in $\mathcal{R}$, which we again denote by the same notation $|\cdot|$. For $t\in \mathbb{N}\cup \{0\}$, set $\delta'\defeq = \frac{1}{T^{nt}}$, 
\[\varepsilon'\defeq (\delta'T^{(t+1)(n-1)})^\frac{1}{n+1}=\frac{T^{\frac{n-1}{n+1}}}{{T^{\frac{t}{n+1}}}}, \text{and } \varepsilon\defeq T^{\beta t}\varepsilon'. \]  From the convergence of  $\displaystyle \sum_{t=0}^{\infty} \psi(q^t)q^{nt}<\infty$, it follows that $\displaystyle \lim_{t\rightarrow \infty} q^{nt}\psi(q^t)=0$, and hence $\psi(q^t)<\frac{1}{q^{nt}}$ for all $t\gg 1$. In view of this, for all but finitely many $t\in \mathbb{N}$, one can clearly see  that the following holds 
\[\mathbf{x}\in \lambda(\mathcal{L}^{<}_t) \Longrightarrow \|g_tu_{\mathbf{x}}\theta\|<|\varepsilon|, \text{ for some }\theta\in \Theta \setminus \{\mathbf{0}\}, \] where $g_t\defeq \diag (\frac{\varepsilon}{\delta'}, \varepsilon, \cdots, \varepsilon, \frac{\varepsilon}{T^{t+1}}, \cdots, \frac{\varepsilon}{T^{t+1}})$,  $u_{\mathbf{x}}$ is the following  $2n\times 2n$ matrix over $F$:
\begin{equation}\label{u_x}\left( \begin{matrix}1 & \mathbf{0} & \mathbf{x} &\tilde{\mathbf{x}}\cdot \mathbf{a} \\ \mathbf{0} &  I_{n-1} & I_{n-1} & \mathbf{a}^t\\  \mathbf{0} & \mathbf{0} & I_{n-1} & \mathbf{0}\\ 0 & \mathbf{0} & \mathbf{0} & 1 \end{matrix}
\right),\end{equation} and \begin{equation}\label{theta}\Theta \defeq \left\{\left(\begin{array}{rcl}p\\ 0\\\cdot \\ \cdot \\ \cdot\\ 0 \\ \mathbf{q}\end{array}\right): p\in \Lambda, \mathbf{q}\in \Lambda^n\right\}.\end{equation} %Similarly for $\lambda(\mathcal{L}^{<}_t(\kappa))$, without any loss in generality, we  take $\kappa=\frac{1}{q^r}$, and observe that, $\forall t\gg 1$, the following holds for any $s\in \mathbb{N}$:  where $g'_t\defeq \diag (T^{(\frac{1}{n+1}-\beta)t+nt+r}, \dots, T^{(\frac{1}{n+1}-\beta)},\cdots , T^{(\frac{1}{n+1}-\beta)},\frac{1}{ T^{t+1}}, \cdots, \frac{1}{T^{t+1}})$, $t\in \mathbb{N}$%. 

We exploit Theorem \ref{qn_gen} to estimate$$\lambda \left(\left \{\mathbf{x}\in U: \|g_tu_{\mathbf{x}}\theta\|<|\varepsilon|, \text{ for some }\theta\in \Theta \setminus \{\mathbf{0}\}\right \}\right).$$  %and 

It is customary to denote the standard basis vectors of $\mathcal{R}^{2n}$ by $$\mathbf{e}_0, \mathbf{e}_{\ast 1}, \dots, \mathbf{e}_{\ast n-1}, \mathbf{e}_1, \dots, \mathbf{e}_n.$$ Let $\mathcal{S}$ be the collection of all $I\subseteq \{0, \ast1, \dots, \ast n-1, 1, \dots, n\}$ such that $\# I\cap \{\ast1, \dots, \ast n-1 \}\leq 1$ and the complement of $\mathcal{S}$ is denoted by $\mathcal{S}'$. Clearly, one has the following direct sum decomposition of $\bigwedge(\mathcal{R}^{2n})$: $\text{span}\{\mathbf{e}_I: I\in \mathcal{S}\}\oplus  \text{span} \{\mathbf{e}_I: I\in \mathcal{S}'\}$. Suppose $\pi$ stands for the projection operator to the first summand with respect to this decomposition. Define $\|\cdot\|: \bigwedge(\mathcal{R}^{2n}) \longrightarrow \mathbb{R}_+$ as follows: for any $\mathbf{v}\in \bigwedge(\mathcal{R}^{2n})$, $\|\mathbf{v}\|$ is defined as the supremum norm of $\pi(\mathbf{v})$. In simple words, if $\mathbf{v}$ is written as a sum of exterior products of basis vectors $\mathbf{e}_i$ and $\mathbf{e}_{\ast i}$, to compute $\|\mathbf{v}\|$ we ignore the components attached with exterior powers that contains some $\mathbf{e}_{\ast i}\wedge \mathbf{e}_{\ast j}$, for some $i\neq j$, and take the supremum norm of the rest. From the construction, the continuity and homogeneity property of $\|\cdot \|$ is clear. It is a routine verification, and hence left to the reader, that for any $\mathbf{v}, \mathbf{w}\in \bigwedge (\mathcal{R}^m)$, one always has the inequality $\|\mathbf{v}\wedge \mathbf{w}\|\leq \|\mathbf{v}\|\cdot \|\mathbf{w}\|$. Thus $\|\cdot \|$ defined as above is submultiplicative. \\

For $t\in\mathbb{N}$, consider $h(x)=g_tu_{\mathbf{x}}\in \GL_{2n}(\mathcal{R})$, where $u_{\mathbf{x}}$ is the matrix defined in \eqref{u_x}. Let $\mathcal{D}\defeq \Lambda$ and $\Theta$ be the $\Lambda$ submodule defined in \eqref{theta}. The discreteness of $\bigwedge^j(\Theta)$ in $\bigwedge^j(\mathcal{R}^{2n})$, for all $j$, ascertains \ref{discrete} of Theorem \ref{qn_gen}. To verify the other two conditions, i.e., \ref{good} and \ref{lower}, we first look at the actions of  $u_{\mathbf{x}}$ on the standard basis vectors:
	\begin{enumerate}  [label=(\roman*)]
		\item $u_{\mathbf{x}}\mathbf{e}_0=\mathbf{e}_0$,
		\item $u_{\mathbf{x}}\mathbf{e}_{\ast i}=\mathbf{e}_{\ast i}$, for all $i=1, \dots, n-1$,
		\item $u_{\mathbf{x}}\mathbf{e}_i=x_i\mathbf{e}_0$+ $\mathbf{e}_{\ast i}+\mathbf{e}_i$, for $1\leq i\leq n-1$; and 
		\item $u_{\mathbf{x}}\mathbf{e}_n=(\alpha_0+\alpha_1x_1+\dots+\alpha_{n-1}x_{n-1})\mathbf{e}_0+\alpha_1 \mathbf{e}_{\ast 1}+ \cdots + \alpha_{n-1}\mathbf{e}_{\ast n-1}+\mathbf{e}_n$.
	\end{enumerate}
It is evident that, for all $\mathbf{w}\in \bigwedge (\Theta)$, each component of $g_tu_{\mathbf{x}}\mathbf{w}$ is a linear polynomial over $\mathcal{R}$. It follows from Proposition \ref{good prop} that all of them are $(C,1/n-1)$-good everywhere, where $C$ is a constant depending only on $n$. Property \ref{good} is now clear from this. So we dedicate ourselves to establish \ref{lower}.\\

In what follows we denote the subspace $\text{span}\{\mathbf{e}_0, \mathbf{e}_1, \dots, \mathbf{e}_n\}$ of $\mathcal{R}^{2n}$ by $\mathcal{R}^{n+1}$. It is obvious that $\Theta$ is discrete in $\mathcal{R}^{n+1}$. Pick $\mathbf{w}=w\,\mathbf{e}_0\wedge \mathbf{e}_1\wedge \cdots \wedge \mathbf{e}_n\in \bigwedge^{n+1}(\Theta)$, where $w\in \mathbb{Z}\setminus \{\mathbf{0}\}$. For any $\mathbf{x}\in U$, the coefficient of $\mathbf{e}_0\wedge \mathbf{e}_{\ast 1}\wedge\mathbf{e}_2\wedge  \cdots \wedge \mathbf{e}_n$ in $\pi(g_tu_{\mathbf{x}}\mathbf{w})$ is easily seen to be $w\,\frac{\varepsilon^{n+1}}{\delta'}\times  \frac{1}{T^{(t+1)(n-1)}}=w\, q^{\beta(n+1)t)}\frac{\varepsilon'^{n+1}}{\delta'T^{(t+1)(n-1)}}=w\, q^{\beta(n+1)t)}$. Therefore \begin{equation}\label{estimate top}\displaystyle \sup_{\mathbf{x}\in U}\|\pi(g_tu_{\mathbf{x}}\mathbf{w})\|\geq |w\, q^{\beta(n+1)t)}|\geq \frac{1}{2}, \forall \mathbf{w}\in \bigwedge^{n+1}(\Theta). \end{equation} \\

 Assume now that $2\leq \ell\leq n$ and $\mathbf{w}\in \bigwedge^{\ell}(\Theta)$. In order to find a lower bound for $\displaystyle \sup_{\mathbf{x}\in U}\|\pi(g_tu_{\mathbf{x}}\mathbf{w})\|$, we proceed as follows. Observe further that,  $\text{span}\{\mathbf{e}_I: I\in \mathcal{S}\}$ can be further decomposed to the direct sum $$\bigwedge (\mathcal{R}^{n+1})\bigoplus \text{span}\{\mathbf{e}_I: I\in \mathcal{S} \text{ and, for some }i, \ast i\in I\}.$$ By $\pi_{\ast}$, we denote the projection operator to the summand $\bigwedge (\mathcal{R}^{n+1})$. It is immediate that, for any $\mathbf{v}\in \bigwedge (\mathcal{R}^{2n})$, $\|\pi(\mathbf{v})\|\geq \|\pi_{\ast}(\mathbf{v})\|$. This inspires us to find a lower bound of $\displaystyle \sup_{\mathbf{x}\in U}\|\pi_{\ast}(g_tu_{\mathbf{x}}\mathbf{w})\|=\sup_{\mathbf{x}\in U}\|\tilde{g}_t\tilde{u_{\mathbf{x}}}\mathbf{w}\|$ instead, where $\tilde{g}_t\defeq  \diag (\frac{\varepsilon}{\delta'}, \frac{\varepsilon}{T^{t+1}}, \dots, \frac{\varepsilon}{T^{t+1}})$, $t\in \mathbb{N}\cup \{0\}$, and $\tilde{u}_{\mathbf{x}}$ is the following  $(n+1)\times (n+1)$ matrix over $F$:
 \[\left( \begin{matrix}1  & \mathbf{x} &\tilde{\mathbf{x}}\cdot \mathbf{a} \\  \mathbf{0}& I_{n-1} & \mathbf{0}\\ 0 & \mathbf{0} & 1 \end{matrix}
 \right).\] Define the functions, for all $1\leq i \leq n-1$, $f_i(\mathbf{x})=x_i$, and $f_n(\mathbf{x})= \tilde{\mathbf{x}}\cdot \mathbf{a}$. Write $\mathbf{f}=(f_1, \dots,f_n)$. Then $\tilde{u}_{\mathbf{x}}$ can be rewritten as $\left(\begin{matrix}
 1& \mathbf{f}(\mathbf{x})\\\mathbf{0} & I_n
 \end{matrix}\right)$. Furthermore, the actions of $\tilde{u}_{\mathbf{x}}$ on the basis vectors of $F^{n+1}$ are as follows:
\begin{enumerate} [label=(\roman*)]
	\item $\tilde{u}_{\mathbf{x}}\mathbf{e}_0=\mathbf{e}_0$, and 
	\item $\tilde{u}_{\mathbf{x}}\mathbf{e}_i=f_i(\mathbf{x})\mathbf{e}_0+\mathbf{e}_i$, for $i=1,\dots,n$. 
	\end{enumerate} From this one obtains that, for any $\mathbf{e}_I\in \bigwedge^j(F^{n+1})$, where $j=0,\dots, n+1$, 
\begin{enumerate} [label=(\alph*)]
	\item $\tilde{u}_{\mathbf{x}}\mathbf{e}_I=\mathbf{e}_I$ if $0\in I$, and 
	\item $\tilde{u}_{\mathbf{x}}\mathbf{e}_I=\mathbf{e}_I+\sum_{i\in I}\pm f_i(\mathbf{x})\mathbf{e}_{I\cup \{0\}\setminus\{i\}}$ otherwise.\end{enumerate} 
	If now $\mathbf{w}=\sum_I w_I\mathbf{e}_I$, then we see that
	\begin{equation}\label{calculation}
\displaystyle \tilde{u}_{\mathbf{x}}\mathbf{w}= \sum_{0\notin I} w_I\mathbf{e}_I+\sum_{0\in I}\left(w_I+\sum_{i\notin I}\pm w_{I\cup\{i\}\setminus\{0\}}f_i(\mathbf{x})\right)\mathbf{e}_I.
	\end{equation}
In \eqref{calculation}, for any $I\ni 0$, it will be useful to put the coefficient of $\mathbf{e}_I$ as $\tilde{\mathbf{f}}(\mathbf{x})\cdot\mathbf{c}_{I,\mathbf{w}}$, where $\tilde{\mathbf{f}}=(1,\mathbf{f})$ and $\mathbf{c}_{I,\mathbf{w}}$ is the following vector:
\[\displaystyle \sum_{i\notin I\setminus \{0\}}\pm w_{I\cup\{i\}\setminus\{0\}} \mathbf{e}_i.  \]Observe that, for any $\mathbf{x}$, \[\tilde{\mathbf{f}}(\mathbf{x})=\tilde{\mathbf{x}}P,\] where $P$ is the block matrix $[I_n|\mathbf{a}^t]$. Thus, the coefficient of $\mathbf{e}_I$ in \eqref{calculation}, for any $I\ni 0$, is rewritten as $\tilde{\mathbf{x}}P\mathbf{c}_{I,\mathbf{w}}$. From the linear independence of $1, f_1, \dots f_{n-1}$ over $F$ on $U$, it follows that the map $\mathbf{v}\mapsto \sup_{\mathbf{x}\in U}\|\tilde{\mathbf{x}}\mathbf{v}\|$ indeed defines a norm on $F^n$, which must be equivalent to the supremum norm on $F^n$, whence for a constant $C'>0$ depending only upon $n$ and $U$, we have $\sup_{\mathbf{x}\in U}\|\tilde{\mathbf{x}}P\mathbf{c}_{I,\mathbf{w}}\|\geq C' \|P\mathbf{c}_{I,\mathbf{w}}\|$, for all $I$. Consequently, one obtains that 
\begin{equation}\label{estimate 1}\sup_{\mathbf{x}\in U}\|\tilde{g}_t\tilde{u}_{\mathbf{x}}\mathbf{w}\|\geq \max \left\{\left(\frac{|\varepsilon|}{q^{t+1}}\right)^{\ell}\max_{0\notin I}|w_I|, C'\frac{|\varepsilon|^{\ell}}{|\delta'|q^{(t+1)(\ell-1)}}\max_{0\in I}\|P\mathbf{c}_{I,\mathbf{w}}\|\right\},
\end{equation} in view of \eqref{calculation}. At this point we need the following lemma:
\begin{lemma}\label{enough} Let $2\leq \ell\leq n+1$ and $\mathbf{w}\in \bigwedge^{\ell} (\Theta)$. Then $\displaystyle \max_{0\in I}\|P\mathbf{c}_{I,\mathbf{w}}\|\geq 1$. 
	\end{lemma}
\begin{proof}This lemma can be proved by adopting the argument of its euclidean version \cite[Lemma 4.6]{Kleinbock-extremal} verbatim. Hence we leave the proof for the pursual of the interested readers. 
\end{proof}From \eqref{estimate 1}, using Lemma \ref{enough}, we deduce with an appropriate choice of $\beta\in (0, \frac{1}{n+1}) \cap \mathbb{Q}$ that,
\begin{equation}\label{estimate middle} \begin{array}{rcl} \sup_{\mathbf{x}\in U}\|g_tu_{\mathbf{x}}\mathbf{w}\|\geq \sup_{\mathbf{x}\in U}\|\tilde{g}_t\tilde{u}_{\mathbf{x}}\mathbf{w}\| \geq & C'\frac{|\varepsilon|^{\ell}}{|\delta'|q^{(t+1)(\ell-1)}}\\ &  =C'\frac{q^{\left(\frac{n-1}{n+1}\right)\ell}}{{q^{\left(\frac{1}{n+1}-\beta\right)t\ell }}}\times q^{nt}\times \frac{1}{q^{(t+1)(\ell-1)}}\\ & \geq C'\frac{q^{2\left(\frac{n-1}{n+1}\right)}}{{q^{\left(\frac{1}{n+1}-\beta\right)nt }}}\times q^{nt}\times \frac{1}{q^{(t+1)(n-1)}}\\ & \geq C'\frac{q^{2\left(\frac{n-1}{n+1}\right)}}{q^{n-1}}\times q^{\left(1-\left(\frac{1}{n+1}-\beta\right)n\right)t } \\ &\geq C'\frac{q^{2\left(\frac{n-1}{n+1}\right)}}{q^{n-1}},\end{array}\end{equation} for all  $\mathbf{w} \in \bigwedge ^{\ell} (\Theta) \text{ and }2\leq \ell\leq n$.\\

Now let $\mathbf{0}\neq \mathbf{w}=p_0\mathbf{e}_0+q_1\mathbf{e_1}+\dots+q_{n-1}\mathbf{e}_{n-1}+q_n\mathbf{e}_n\in \bigwedge^1(\Theta)$. So, for any $\mathbf{x}\in U$, 
\begin{eqnarray*}g_tu_{\mathbf{x}}\mathbf{w} & =\frac{\varepsilon}{\delta'}((p_0+q_n\alpha_0)+ (q_1+q_n\alpha_1)x_1+\cdots+(q_{n-1}+q_n\alpha_{n-1})x_{n-1})\mathbf{e}_0\\&+ \varepsilon(q_1+q_n\alpha_1)\mathbf{e}_{\ast 1}+\cdots  + \varepsilon(q_{n-1}+q_n\alpha_{n-1})\mathbf{e}_{\ast n-1}\\ & + \frac{\varepsilon}{T^{t+1}}q_1\mathbf{e}_1+\cdots \frac{\varepsilon}{T^{t+1}}q_n\mathbf{e}_n.\end{eqnarray*} The Diophantine condition given by \eqref{cond} ensures that 
\[\sup_{\mathbf{x}\in U}|(p_0+q_n\alpha_0)+ (q_1+q_n\alpha_1)x_1+\cdots+(q_{n-1}+q_n\alpha_{n-1})x_{n-1}|\geq C''\frac{1}{|q_n|^{n-\delta}},\] for some constant $C''>0$ depending on $U$. From this, we see that 
\begin{equation}\label{estimate 3}
\sup_{\mathbf{x}\in U}\|g_tu_{\mathbf{x}}\mathbf{w}\|\geq\max\left\{C''\frac{|\varepsilon|}{|\delta'|}\frac{1}{|q_n|^{n-\delta}},\frac{|\varepsilon|}{q^{t+1}}|q_n |\right\}.
\end{equation}
Hence we have to seek the solution of the equation $C''\frac{|\varepsilon|}{|\delta'|}\frac{1}{y^{n-\delta}}=\frac{|\varepsilon|}{q^{t+1}}y $. The unique positive solution of this equation is $y_0=C''^{\frac{1}{n-\delta+1}}\left(\frac{q^{t+1}}{|\delta'|}\right)^{\frac{1}{n-\delta+1}}=C''^{\frac{1}{n-\delta+1}}q^{\frac{1}{n-\delta+1}}q^{\frac{(n+1)t}{n-\delta+1}}$. This yields that, with a further refined choice of $\beta$ if necessary,  
\begin{equation}\label{estimate 4}\begin{array}{rcl}\displaystyle \sup_{\mathbf{x}\in U}\|g_tu_{\mathbf{x}}\mathbf{w}\|&\geq C''^{\frac{1}{n-\delta+1}}q^{\frac{1}{n-\delta+1}}q^{\frac{(n+1)t}{n-\delta+1}}\times \frac{1}{q^{t+1}}\times\frac{q^{\frac{n-1}{n+1}}}{q^{\left(\frac{1}{n+1}-\beta\right)t }}\\ & \geq C''^{\frac{1}{n-\delta+1}}q^{\left(\frac{1}{n-\delta+1}-1\right)}q^{\left(\left(\frac{n+1}{n+1-\delta}-1\right)-\left(\frac{1}{n+1}-\beta\right)\right)t}\\& \geq C''^{\frac{1}{n-\delta+1}}q^{\left(\frac{1}{n-\delta+1}-1\right)},
\end{array}
\end{equation}  for all $\mathbf{0}\neq \mathbf{w}\in \bigwedge^1(\Theta)$.
Summarizing the observations \eqref{estimate top}, \eqref{estimate middle} and \eqref{estimate 4}, we ascertain \ref{lower} with $\mu=\lambda$ and $\rho$ equal to the following explicit constant:
\begin{equation}\label{rho}\min \left\{\frac{1}{2}, C'\frac{q^{2\left(\frac{n-1}{n+1}\right)}}{q^{n-1}},  C''^{\frac{1}{n-\delta+1}}q^{\left(\frac{1}{n-\delta+1}-1\right)}\right\}.\end{equation}
Now the Theorem \ref{qn_gen} applies and provides us with the following estimate:
\begin{eqnarray*}\lambda \left(\left \{\mathbf{x}\in U: \|g_tu_{\mathbf{x}}\theta\|<|\varepsilon|, \text{ for some }\theta\in \Theta \setminus \{\mathbf{0}\}\right \}\right)\leq &  (n+1)C(N_{F^{n-1}}D_{\lambda}^2)^{n+1} \left(\frac{|\varepsilon|}{\rho}\right)^{\frac{1}{n-1}}\lambda(U) \\ \leq &  (n+1)C(N_{F^{n-1}}D_{\lambda}^2)^{n+1} \frac{q^{\frac{1}{n+1}}}{\rho^{\frac{1}{n-1}}}\frac{1}{q^{\left(\frac{\frac{1}{n+1}-\beta}{n-1}\right)t }}\lambda(U);
	\end{eqnarray*} which in turn implies the desired convergence of $ \displaystyle \sum_{t=0}^{\infty}\lambda(\mathcal{L}^{<}_t)$.\\

We now turn to the quantitative counterpart of the above. Exactly same way to that of $\lambda(\mathcal{L}^{<}_t)$, we shall provide an estimate of $\lambda(\mathcal{L}^{<}_t(\kappa))$, for $t\geq 0$, in the following section. 
\section{Proof of \eqref{quant small}} Let $\beta, \mathcal{R}, u_{\mathbf{x}}$,  where $\mathbf{x}\in U$, and $\Theta$ be as given in \S \ref{khinchine}. To avoid using too many different notations, we retain some of those from \S \ref{khinchine} in what follows, though the expressions that they stand for from now on are going to be different from the previous ones   in many cases. Without any loss in generality we can consider $\kappa=\frac{1}{q^r}$, for $r\in \mathbb{N}$. Now we set the following for all $t\in \mathbb{N}\cup\{0\}$: $\delta'\defeq \frac{1}{T^{nt+r}}$ so that $|\delta'|=\frac{\kappa}{q^{nt}}$, and 
\[\varepsilon'\defeq \left(\delta'T^{(t+1)(n-1)}\right)^\frac{1}{n+1}=\kappa^{\frac{1}{n+1}}\frac{T^{\frac{n-1}{n+1}}}{{T^{\frac{t}{n+1}}}} \text{ and }\varepsilon\defeq T^{\beta t}\varepsilon'. \]
  Define $g_t\defeq \diag (\frac{\varepsilon}{\delta'}, \varepsilon, \cdots, \varepsilon, \frac{\varepsilon}{T^{t+1}}, \cdots, \frac{\varepsilon}{T^{t+1}})\in \GL_{2n}(\mathcal{R})$. Similar to that of $\mathcal{L}_t^{<}$, one can immediately see that, for all $t\geq 0$, 
 \begin{equation}\mathcal{L}^{<}_t(\kappa) \subseteq \left\{\mathbf{x}\in U: \|g_tu_{\mathbf{x}}\theta\|<|\varepsilon|, \text{ for some }\theta\in \Theta \setminus \{\mathbf{0}\}\right\}. \end{equation} As usual, we use Theorem \ref{qn_gen} to estimate the measure of set appearing in the RHS of \eqref{set}. With same analysis to that of \S\ref{khinchine}, we note the following:\vspace{0.1cm}
 \begin{enumerate}
 	\item \label{quant-good}For all $\mathbf{w}\in \bigwedge (\Theta)$, each component of $g_tu_{\mathbf{x}}\mathbf{w}$ is a linear polynomial over $\mathcal{R}$, and hence by proposition \ref{good prop}, all of them are $(C,1/n-1)$-good everywhere, where $C$ is a constant depending only on $n$.
 	\item \label{quant top}$\displaystyle \sup_{\mathbf{x}\in U}\|\pi(g_tu_{\mathbf{x}}\mathbf{w})\|\geq |w\, q^{\beta(n+1)t)}|\geq \frac{1}{2}, \forall \mathbf{w}\in \bigwedge^{n+1}(\Theta)$.
 	\item Let  $2\leq \ell\leq n$. For all  $\mathbf{w} \in \bigwedge ^{\ell} (\Theta)$, 
 	 $\sup_{\mathbf{x}\in U}\|g_tu_{\mathbf{x}}\mathbf{w}\|$ is at least 
 	\[\begin{array}{rcl} C'\frac{|\varepsilon|^{\ell}}{|\delta'|q^{(t+1)(\ell-1)}}& =C'\kappa^{\frac{\ell}{n+1}}\frac{q^{\left(\frac{n-1}{n+1}\right)\ell}}{{q^{\left(\frac{1}{n+1}-\beta\right)t\ell }}}\times \frac{q^{nt}}{\kappa}\times \frac{1}{q^{(t+1)(\ell-1)}}\\ & \geq C'\kappa^{\left(\frac{\ell}{n+1}-1\right)}\frac{q^{2\left(\frac{n-1}{n+1}\right)}}{{q^{\left(\frac{1}{n+1}-\beta\right)nt }}}\times q^{nt}\times \frac{1}{q^{(t+1)(n-1)}}\\ &\geq C'\frac{q^{2\left(\frac{n-1}{n+1}\right)}}{q^{n-1}}\times q^{\left(1-\left(\frac{1}{n+1}-\beta\right)n\right)t } \\ & \geq C'\frac{q^{2\left(\frac{n-1}{n+1}\right)}}{q^{n-1}}.\end{array}\]
 	\item \label{quant one}For all $\mathbf{0}\neq \mathbf{w}\in \bigwedge^1(\Theta)$, 
 	\[\begin{array}{rcl}\displaystyle \sup_{\mathbf{x}\in U}\|g_tu_{\mathbf{x}}\mathbf{w}\|&\geq C''^{\frac{1}{n-\delta+1}}\left(\frac{q}{\kappa}\right)^{\frac{1}{n-\delta+1}}q^{\frac{(n+1)t}{n-\delta+1}}\times \frac{1}{q^{t+1}}\times \kappa^{\frac{1}{n+1}}\frac{q^{\frac{n-1}{n+1}}}{q^{\left(\frac{1}{n+1}-\beta\right)t }}\\ & \geq C''^{\frac{1}{n-\delta+1}}q^{\left(\frac{1}{n-\delta+1}-1\right)}\kappa^{
 	\frac{-\delta}{(n+1)(n+1-\delta)}}q^{\left(\left(\frac{n+1}{n+1-\delta}-1\right)-\left(\frac{1}{n+1}-\beta\right)\right)t}\\& \geq C''^{\frac{1}{n-\delta+1}}q^{\left(\frac{1}{n-\delta+1}-1\right)},
 	\end{array}\]
 	\end{enumerate}	The properties \ref{good} and \ref{lower} in Theorem \ref{qn_gen} now follow immediately from \eqref{quant-good} and \eqref{quant top}-\eqref{quant one} respectively; indeed the choice of $\rho$ that we have made in \eqref{rho} works here too. The Property \ref{discrete} is anyway clear from the discreteness of $\bigwedge^j(\Theta)$ in $\bigwedge^j(\mathcal{R}^{2n})$, for all $j$. Hence we obtain that 
 \begin{eqnarray*}\lambda \left(\left \{\mathbf{x}\in U: \|g_tu_{\mathbf{x}}\theta\|<|\varepsilon|, \text{ for some }\theta\in \Theta \setminus \{\mathbf{0}\}\right \}\right)\leq &  (n+1)C(N_{F^{n-1}}D_{\lambda}^2)^{n+1} \left(\frac{|\varepsilon|}{\rho}\right)^{\frac{1}{n-1}}\lambda(U) \\ \leq &  (n+1)C(N_{F^{n-1}}D_{\lambda}^2)^{n+1} \frac{q^{\frac{1}{n+1}}}{\rho^{\frac{1}{n-1}}}\kappa^{\frac{1}{n^2-1}}\frac{1}{q^{\left(\frac{\frac{1}{n+1}-\beta}{n-1}\right)t }}\lambda(U).
 \end{eqnarray*} 
Denote $K_0\defeq \displaystyle (n+1)C(N_{F^{n-1}}D_{\lambda}^2)^{n+1} \frac{q^{\frac{1}{n+1}}}{\rho^{\frac{1}{n-1}}}$ and $K_1\defeq \displaystyle \sum_{t=0}^{\infty}\frac{1}{q^{\left(\frac{\frac{1}{n+1}-\beta}{n-1}\right)t }}$. Thus one has 
\[\displaystyle \sum_{t=0}^{\infty}\lambda(\mathcal{L}_t^{<})\leq K_0K_1\kappa^{\frac{1}{n^2-1}}\lambda(U)<\frac{\xi}{2}\lambda(U).\] The proof of \eqref{quant small} is now complete. $\Box$

\end{document}